\documentclass[10pt]{article}
\usepackage[preprint]{tmlr}

\RequirePackage{amsthm,amsmath,amsfonts,amssymb}
\usepackage{color}
\usepackage{booktabs} 
\usepackage{multirow}
\usepackage{ifthen}
\usepackage{url}
\usepackage{subcaption}
\usepackage{enumitem}

\usepackage{tikz} 
\usetikzlibrary{positioning, arrows.meta, shapes.geometric, shapes.misc, fit}
\tikzset{%
  semithick,
  >={Stealth[width=1.5mm,length=2mm]},
  obs1/.style = {name = #1, circle, inner sep = 8pt, label = center:$#1$},
  obs/.style 2 args = {name = #1, circle, inner sep = 8pt, label = center:$#2$}
}
\usepackage{caption}
\usepackage[aboveskip=-1pt,belowskip=-1pt]{subcaption}

\theoremstyle{plain}
\newtheorem{theorem}{Theorem}[section]

\newtheorem{corollary}[theorem]{Corollary}
\theoremstyle{definition}
\newtheorem{definition}[theorem]{Definition}

\newcommand{\cond}{\,\vert\,}

\newcommand{\E}{\textrm{E}}
\newcommand{\sd}{\textrm{sd}}
\newcommand{\Cor}{\textrm{Cor}}
\newcommand{\eqs}{\!=\!}
\newcommand\indep{\protect\mathpalette{\protect\independenT}{\perp}}
\def\independenT#1#2{\mathrel{\rlap{$#1#2$}\mkern2mu{#1#2}}}

\DeclareMathOperator{\paOP}{pa}
\newcommand{\pa}[1][]{%
\ifthenelse{ \equal{#1}{} }
{\paOP}
{\paOP_{#1}}
}
\DeclareMathOperator{\chOP}{ch}
\newcommand{\ch}[1][]{%
\ifthenelse{ \equal{#1}{} }
{\chOP}
{\chOP_{#1}}
}
\DeclareMathOperator{\deOP}{de}
\newcommand{\de}[1][]{%
\ifthenelse{ \equal{#1}{} }
{\deOP}
{\deOP_{#1}}
}
\DeclareMathOperator{\anOP}{an}
\newcommand{\an}[1][]{%
\ifthenelse{ \equal{#1}{} }
{\anOP}
{\anOP_{#1}}
}

\newcommand{\+}[1]{\ensuremath{\mathbf{#1}}}
\newcommand{\sG}{\mathcal{G}}
\newcommand{\sM}{\mathcal{M}}

\title{Multiple Imputation Guided by Full Law and Target Law Identifiability under the Missing Not At Random Assumption}

\author{
\name Juha Karvanen \email juha.t.karvanen@jyu.fi \\
      \addr Department of Mathematics and Statistics\\
      University of Jyvaskyla, Finland
      \AND
\name Santtu Tikka \email santtu.tikka@jyu.fi \\
      \addr Department of Mathematics and Statistics\\
      University of Jyvaskyla, Finland
}

\begin{document}
\maketitle

\begin{abstract}
The central challenges in missing data models concern the identifiability of two distributions: the target law and the full law. The target law refers to the joint distribution of the data variables, whereas the full law refers to the joint distribution of the data variables and their corresponding response indicators. However, the relationship between the identifiability of these two distributions and the feasibility of multiple imputation has not been clearly established when the data are missing not at random (MNAR). We present a procedure in which the choice of imputation method is guided by identifiability considerations. The identifiability of the full law implies the applicability of conditionally complete imputation methods that draw imputations for all missing data patterns. By contrast, the non-identifiability of the full law implies that any multiple imputation method aiming to implement conditionally complete imputation will produce biased estimates, thereby also restricting the options for estimating the target law.
We demonstrate that alternative imputation strategies can sometimes enable the estimation of the target law in such cases. Specifically, we introduce factorizable imputation where certain observed values are also imputed and the imputed data are weighted in the analysis. 
\end{abstract}

\section{Introduction} \label{sec:intro}
Multiple imputation is a popular tool for practical data analysis with missing data \citep{vanbuuren2018flexible}. A typical imputation method keeps observed values unchanged and replaces all missing values with imputed values producing many imputed versions of the full data. The theoretical basis of multiple imputation \citep{rubin1987multiple} has been built on the missing at random (MAR) assumption \citep{rubin1976inference,little2019statistical,seaman2013meant}. This assumption is not strictly necessary and multiple imputation has been applied also in some cases where the data are missing not at random (MNAR) \citep{galimard2016multiple, hammon2020multiple, beesley2021multiple, munoz2024multiple}.

Alongside the MAR assumption, there has been growing interest in missing data literature in more detailed nonparametric assumptions and their implications for identifiability, also known as recoverability. These assumptions are often represented by graphical missing data models \citep{doi:10.1177/0962280210394469, NIPS2013_0ff8033c,Karvanen2015studydesign}. The identifiability of the full law, i.e., the joint distribution of data variables and response indicators, and the identifiability of the target law, i.e., the joint distribution of data variables, have been among the main questions of interest \citep[e.g.,][]{pmlr-v77-tian17a,pmlr-v115-bhattacharya20b, Mohan_2021, pmlr-v216-guo23a, tikka2024monotone}. The full law and the target law are identifiable also in some cases that are classified as MNAR. Currently, the sufficient and necessary conditions for the nonparametric identifiability are known for the full law  \citep{pmlr-v119-nabi20a}, but not for the target law.

Multiple imputation and identifiability of the full law have not been explicitly linked in the literature. \citet{mathur2024imputation} give a graphical condition for an imputation model trained on complete cases to be valid. However, imputation methods such as multiple imputation by chained equations \citep{vanbuuren2011mice} use also partially observed cases, which extends the applicability of imputation. 

In this paper, we study the consequences of the full law and the target law identifiability for multiple imputation. We define an imputation method as a collection of stochastic kernels and call an imputation method conditionally complete if for every missingness pattern, the imputation kernel equals the conditional distribution of the missing data. In practice, the conditional completeness means that  imputations can be drawn from the correct conditional distributions for all possible missing data patterns. We show that a conditional complete imputation method exists if and only if the full law is identifiable. This result is intuitive but has not been presented in the literature. 
 We consider only correctly specified imputation models and ignore the challenges related to their estimation. The link between multiple imputation and the identifiability of the full law implies that multiple imputation is applicable also in some MNAR cases. The result means that we can evaluate the potential of conditional complete multiple imputation in a MNAR case by checking the sufficient and necessary conditions \citep{pmlr-v119-nabi20a} for nonparametric identifiability of the full law.

There are also cases where the full law is not identifiable, and consequently, any multiple imputation method that aims to implement conditionally complete imputation is doomed to produce biased estimates. We show that in such cases, alternative imputation strategies that are not conditionally complete can sometimes be valid, for example when the target law is identifiable. When the target law factorizes into identifiable components, we show that a special imputation method which we call factorizable imputation yields valid imputations. The factorization mandates two atypical practices: replacing actual observations by imputed values under certain conditions and weighting the imputed data in the analysis.

The rest of the paper is organized as follows. In Section~\ref{sec:mainresult}, the necessary notation is introduced and the results on identifiability and multiple imputation are presented. Section~\ref{sec:procedure} presents a practical procedure for multiple imputation guided by identifiability considerations. The implications of the theoretical results are demonstrated via examples in Section~\ref{sec:examples}. Section~\ref{sec:conclusion} concludes the paper.

\section{Theoretical results} \label{sec:mainresult}

\subsection{Notation and definitions}
Let $\+ X$ denote the set of data variables that may be missing (called partially observed variables). Under missing data, we do not observe these variables directly. Instead, we observe a set of proxy variables $\+ X^\ast$ and their response indicators $\+ R$. There may also be variables that are fully observed, denoted by $\+ O$. The proxy variables $X^\ast \in \+ X^\ast$ are defined as follows:
\begin{equation} \label{eq:missingness_mechanism}
  X^\ast = \begin{cases}
    X & \text{if } R_X = 1, \\
    \text{NA} & \text{if } R_X = 0,
  \end{cases}
\end{equation}
where a missing value is indicated by \text{NA} (not available). In other words, the proxies represent the true values of the partially observed variables when they are not missing ($R_X = 1$), and the value is missing otherwise ($R_X = 0$). We generalize the notation of response indicators to sets $\+ Y \subseteq \+ X$ as follows. $R_{\+ Y} = 1$ denotes that all individual response indicators $R_Y$ such that $Y \in \+ Y$ have the value assignment $R_Y = 1$. The notation $R_{\+ Y} = 0$ is analogous.

A set of distributions over a set of random variables $\+ V$ that can be partitioned into the sets $\+ X$, $\+ X^\ast$, $\+ R$, and $\+ O$ is called a missing data model. When considering a particular value assignment of $\+ r$ to $\+ R$, we further partition $\+ X$ into the sets $\+X_1$ and $\+X_0$, where variables $\+X_1$ are observed and variables $\+X_0$ are missing under this particular value assignment, respectively. In other words, when $\+ r$ is fixed, we have that $R_{X} = 1$ for all $X \in \+X_1$, and $R_X = 0$ for all $X \in \+X_0$.

In missing data models, the joint distribution of $\+ V$ factorizes as follows 
\[
  p(\+ V) = p(\+ O, \+ X^\ast, \+ X, \+ R) = p(\+ X^\ast|\+ X, \+ R)p(\+ O, \+ X, \+ R),
\]
where $p(\+ O, \+ X, \+ R)$ is the \emph{full law}, and $p(\+ X^\ast|\+ X, \+ R)$ represents the deterministic component related to \eqref{eq:missingness_mechanism}. The full law is characterized by the \emph{target law} $p(\+ O, \+ X)$ and the \emph{missingness mechanism} $p(\+ R|\+ O, \+ X)$. Finally, the distribution $p(\+ O, \+ X^\ast, \+ R)$ that is available for practical analyses is called the \emph{observed data law}.

A nonparametric missing data model $\sM_\Omega$ is a missing data model whose members share a description $\Omega$. The description is a set of nonparametric assumptions about the missing data model, such as (counterfactual) conditional independence restrictions for variables $\+V$. These restrictions can be listed directly or they can be described by a graphical model, for example. Missing completely at random (MCAR) is an example of a description $\Omega$ that defines a nonparametric missing data model $\sM_\Omega$.

We focus on graphical missing data models depicted by directed acyclic graphs (DAGs) \citep{doi:10.1177/0962280210394469,NIPS2013_0ff8033c,Karvanen2015studydesign}. In other words, the description $\Omega$ encodes a missing data model such that the joint distribution factorizes according to the structure of the DAG, i.e.,
\[
  p(\+ V) = \prod_{V_i \in \+ V} p(V_i \cond \pa[\sG](V_i)),
\]
where $\sG$ is a DAG and $\pa[\sG](V_i)$ denotes the parents of $V_i$ in $\sG$. We also assume positivity for the full law, i.e., $p(\+ O = \+ o, \+ X = \+ x, \+ R = \+ r) > 0$ for all value assignments $(\+ o, \+ x, \+ r)$ to $(\+ O, \+ X, \+ R)$. This means that there are no deterministic relationships between the variables in the missing data model, except those related to the definition of proxy variables in equation~\eqref{eq:missingness_mechanism}. Conditional independence restrictions implied by the DAG can be derived using d-separation \citep{pearl_1995}.
To simplify the presentation, we assume in the rest of the paper  that all variables may have missing values, i.e.,  $\+O=\emptyset$. Accordingly, $p(\+X,\+R)$ will stand for the full law, $p(\+X)$ for the target law, and $p(\+X^*,\+ R)$ for the observed data law. 

We define a general notion of identifiability that is not specific to any particular description $\Omega$ of a nonparametric missing data model. 
\begin{definition} 
An estimand or a probabilistic query $Q$ is said to be \emph{identifiable} from $(\Omega, p( \+ X^\ast, \+ R))$ if there exists a functional $f$ such that $Q = f(p( \+ X^\ast, \+ R))$ for every element of a nonparametric missing data model $\sM_\Omega$. 
\end{definition}
\noindent In this paper we focus on the identifiability of the target law and the full law meaning that our queries of interest are $Q = p(\+ X)$ and $Q = p(\+ X, \+ R)$, respectively.

An imputation method turns an incomplete dataset into a complete dataset \citep[chapter~2.3.3]{vanbuuren2018flexible}. We characterize this notion in the following definition. 
\begin{definition} \label{def:imputationmethod}
  Let $\sM_\Omega$ be a nonparametric missing data model with observed data law $p( \+ X^\ast, \+ R)$. An \emph{imputation method} is a collection of stochastic kernels, called \emph{imputation kernels}
  \[
    \xi_{\Omega, p(\+ X^\ast, \+ R)} = \left\{ \xi_{\+ r}(\+ x_0 \cond \+ x_1) : \+ r \in \{0, 1\}^{|\+ X|} \right\},
  \]
  where, for each missingness pattern $\+ r$, the variables are partitioned into observed variables $\+ X_1$ and missing variables $\+ X_0$. The kernel $\xi_{\+ r}(\+ x_0 \cond \+ x_1)$ is a conditional distribution over the missing variables $\+ X_0$, given the observed variables $\+ X_1 = \+ x_1$, whose form is determined only by $\Omega$ and the observed data law $p(\+ X^\ast, \+ R)$.
\end{definition}
In practice, imputations are drawn from the imputation kernels.
Naturally, not all imputation methods correspond with the true conditional distribution of $\+ X_0$ for all partitions of $\+ X$. 

\subsection{Conditionally complete imputation and the full law}
\citet{rubin1987multiple} gives a formal definition of multiple imputation. Following these ideas, we define a conditionally complete imputation method as follows.
\begin{definition} \label{def:imputation}
Let $\sM_\Omega$ be a nonparametric missing data model with observed data law $p(\+ X^\ast, \+ R)$. An imputation method $\xi_{\Omega, p( \+ X^\ast, \+ R)}$ is \emph{conditionally complete} if, for every missingness pattern $\+ r$, the corresponding imputation kernel satisfies
\begin{equation} \label{eq:conditionalpredictive}
\xi_{\+ r}(\+ x_0 \cond \+ x_1) = p(\+X_0 \cond \+X_1 = \+x_1, R_{\+X_1}=1, R_{\+X_0}=0),
\end{equation}
for all possible value assignments $(\+ x_0, \+ x_1)$ to $(\+ X_0, \+ X_1)$.
\end{definition}

In practice, many imputation methods estimate the conditional predictive distributions and then simulate values from the estimated models also known as imputation models.  The imputation model may contain interactions and nonlinear relationships. In MAR settings, $\+X_0$ in equation~\eqref{eq:conditionalpredictive} is independent from $\+R$ but in MNAR settings the response indicators are usually needed as predictors. In commonly used imputation methods, the NA values are replaced with multiple imputed values but actual observations are not altered. Removing rows of the data or changing observed values to missing values are in general considered dubious operations that lead to information loss and may cause bias. This may sound obvious, but later we will also consider special imputation methods where valid imputations cannot be produced unless these principles are violated.

With the definitions above, we are ready to present the first result of the paper:
\begin{theorem} \label{thm:fulllawmi}
  Let $\sM_\Omega$ be a nonparametric missing data model. A conditionally complete imputation method $\xi_{\Omega, p( \+ X^\ast, \+ R)}$ exists if and only if the full law $p( \+X, \+ R)$ is identifiable from $(\Omega, p( \+ X^\ast, \+ R))$.
\end{theorem}
\begin{proof}
Assume first that $p(\+X, \+ R)$ is identifiable. Then any conditional distribution of the form of equation~\eqref{eq:conditionalpredictive} can be derived from the joint distribution $p(\+X, \+ R)$ and is thus identifiable. The positivity of the conditional distributions follows from the positivity of the full law which we assume. Consequently, the stochastic kernels required for the conditionally complete imputation method can be constructed based on the identifying functional $f(p(\+ X^\ast, \+ R)) = p(\+ X, \+ R)$. 

Assume then that a conditionally complete imputation method $\xi_{\Omega,p(\+ X^\ast, \+ R)}$ exists. The full law for the case where $\+X_1$ is observed and $\+X_0$ is missing can be then written as
\begin{align*}
&p(\+X_1, \+X_0,  R_{\+X_1}=1, R_{\+X_0}=0) 
= p( \+X_1, R_{\+X_1}=1, R_{\+X_0}=0)
p(\+X_0 \cond \+X_1, R_{\+X_1}=1, R_{\+X_0}=0).
\end{align*}
Now, $p( \+X_1, R_{\+X_1}=1, R_{\+X_0}=0)=p(R_{\+X_1}=1, R_{\+X_0}=0) p( \+X^*_1 \cond R_{\+X_1}=1, R_{\+X_0}=0)$ is directly identifiable from the observed data law $p( \+ X^\ast, \+ R)$, and $p(\+X_0 \cond \+X_1, R_{\+X_1}=1, R_{\+X_0}=0)$ is obtained from the imputation kernels of the conditionally complete imputation method $\xi_{\Omega, p(\+ X^\ast, \+ R)}$. Since the partition of $\+ X$ into $\+ X_0$ and $\+ X_1$ is arbitrary, it follows that the full law $p(\+X, \+ R)$ is identifiable.
\end{proof}

Sections~\ref{exsec:simplemnar}, \ref{exsec:mediatedselfcensoring} and \ref{exsec:or} provide examples where the data are MNAR and the full law identifiable. It follows from Theorem~\ref{thm:fulllawmi} that conditionally complete imputation method exists and it is demonstrated in Section~\ref{sec:twovariablesimulation} that multiple imputation can produce unbiased estimates in these examples. 

Theorem~\ref{thm:fulllawmi} has implications for cases where the target law is identifiable but the full law is not. Consider an imputation method with imputation kernels $\xi_{\+ r}(\+ x_0 \cond \+ x_1)$ derived from the observed data law. When the full law is not identifiable, the imputation kernel violates the condition of equation~\eqref{eq:conditionalpredictive} at least for one missingness pattern $\+r'$. Consequently, the imputation method does correctly recover the distribution $p(\+X \cond \+R=\+r')$. The target law can be written as
\[
p(\+X) = \sum_{\+r} p(\+X \cond \+R=\+r) p(\+R=\+r).
\] 
It can be seen from this representation that if the imputation method does not use the correct imputation kernel for $p(\+X \cond \+R=\+r')$, the target law can be recovered only if the violations in $p(\+X \cond \+R=\+r')$ are compensated by the imputation kernels for other missingness patterns. Devising such an imputation method appears challenging if not impossible. In practice, it is unlikely that a typical imputation method would automatically fulfill the required condition. This motivates the search for special imputation methods in cases where the target law is identifiable but the full law is not.

Sections~\ref{exsec:colluder} and \ref{exsec:independent} as well as Sections~\ref{exsec:factorizable}--\ref{exsec:othermeans} provide examples where the full law is not identifiable although the target law is. From Theorem~\ref{thm:fulllawmi} we learn that a conditionally complete imputation method does not exist in these examples, and the simulation studies in Section~\ref{sec:examples} demonstrate the failure of imputation methods that unsuccessfully aim for conditional completeness.

\subsection{Factorizable imputation and the target law}

Alternative, not conditionally complete imputation methods are sometimes valid for the target law estimation even if the full law is not identifiable.
Next, we will present two such methods. The first method can be seen as an application of admissible factorization by \citet{NIPS2013_0ff8033c} and the second one is based on a new theoretical result that we will introduce. The second method is more general than the first one but the first one is simpler to implement.

In the both methods, we will consider ordered variables. When an ordering is fixed and $Z$ precedes $W$ in the ordering, we write $Z \prec W$. We use subscripts within parentheses to denote variables in an ordering, e.g., $X_{(1)} \prec X_{(2)} \prec \cdots \prec X_{(n)}$. We also introduce the following notation for the set of predecessors:
\[
  \+ X_{\prec (i)} = \{X_{(j)} \in \+ X \mid X_{(j)} \prec X_{(i)}\},
\]
with the convention that $\+ X_{\prec(1)} = \emptyset$. In other words, $\+ X_{\prec (i)}$ denotes the set of variables that precede $X_{(i)}$ in the ordering. Our goal is to construct an imputation method that leverages a specific ordering of the variables in $\+ X$ to draw correct imputations. The following definition describes such a method, which we call \emph{directly factorizable imputation}.

\begin{definition} \label{def:factorizable_imputation}
Let $\sM_\Omega$ be a nonparametric missing data model and let the variables in $\+ X$ be ordered as ${X_{(1)} \prec \cdots \prec X_{(n)}}$. 
An imputation method is \emph{directly factorizable} with respect to the ordering if its imputation kernels are of the form
\begin{equation} \label{eq:df_kernels}
  \xi^{\mathrm{DF}}_{\+ r}(\+ t \cond \+ h) = \prod_{j = m(\+ r)}^n p(X_{(j)} \cond \+ X_{\prec(j)}),
\end{equation}
where $m(\+ r)$ is the index of the first variable with a missing value in the ordering with respect to the missingness pattern $\+ r$, $\+ H(\+ r) = \{X_{(j)} : j < m(\+ r)\}$, $\+ T(\+ r) = \{X_{(j)} : j \geq m(\+ r)\}$, and $\+ h$ denotes the values of $\+ H(\+ r)$ and $\+ t$ denotes the values of $\+ T(\+ r)$.
\end{definition}

Note that a directly factorizable imputation method may impute $X_{(k)}$ even if $R_{X_{(k)}} = 1$ 
in cases where $R_{X_{(j)}} = 0$ for some $1 < j < k$. In practice, the observed value is then replaced by the imputed value; an operation that is not part of a typical implementation of multiple imputation. In other words, the missingness pattern is forced to be monotone. This practice may sound counter-intuitive as it leads to loss of information, but it is required to ensure that the imputations are drawn from the correct conditional distributions. However, we note that the loss of information is not as substantial as in complete case analysis, because information up to the first missing value is still available for analysis. The next theorem provides a sufficient criterion for the existence  of a directly factorizable imputation method.

\begin{theorem} \label{thm:factorizable_valid}
Let $\sM_\Omega$ be a nonparametric missing data model. If there exists an ordering ${X_{(1)} \prec \cdots \prec X_{(n)}}$ of the variables in $\+X$ such that $X_{(1)} \indep R_{(1)}$ and ${X_{(j)} \indep R_{X_{(j)}}, R_{ \+ X_{\prec (j)}} \cond \+ X_{\prec (j)}}$ for all $j = 2,\ldots,n$, then a directly factorizable imputation method exists.
\end{theorem}
\begin{proof}
Using the assumed conditional independence restrictions, we may write
\[
  \begin{aligned}
  \prod_{j = 1}^n p(X_{(j)} \cond \+ X_{\prec (j)})
    &= \prod_{j = 1}^n p(X_{(j)} \cond \+ X_{\prec (j)}, R_{X_{(j)}} = 1, R_{\+ X_{\prec (j)}} = 1) \\
    &=\prod_{j = 1}^n p(X^*_{(j)} \cond \+ X^*_{\prec (j)}, R_{X_{(j)}} = 1, R_{\+ X_{\prec (j)}} = 1)
  \end{aligned}
\]
Each term in the final right-hand side expression is a function of the observed data distribution $p( \+ X^*, \+ R)$. Thus, a directly factorizable imputation method exists.
\end{proof}
In practice, finding a suitable ordering required by Theorem~\ref{thm:factorizable_valid} may be challenging for large models because the number of possible orderings is the factorial of $n$. To alleviate this, we note that the target law is identifiable under the conditions of Theorem~\ref{thm:factorizable_valid}; a special case of identification via an admissible factorization presented by \citet{NIPS2013_0ff8033c}, and thus their results can be used to reduce the search space of possible orderings in some cases. 

\begin{corollary} \label{cor:tagerlaw_id}
  Let $\sM_\Omega$ be a nonparametric missing data model. The target law $p( \+ X)$ is identifiable under the conditions of Theorem~\ref{thm:factorizable_valid}.
\end{corollary}
\begin{proof}
By writing
\[
  p( \+X) = \prod_{j = 1}^n p(X_{(j)} \cond \+ X_{\prec (j)})
\]
we can see that the right-hand side is identifiable following the proof of Theorem~\ref{thm:factorizable_valid} because each term is a function of the observed data law and thus identifiable.
\end{proof}

Sections~\ref{exsec:colluder} and \ref{exsec:factorizable} provide examples where directly factorizable imputation can be applied although the full law is not identifiable. In the example of Section~ \ref{exsec:simplemnar}, both directly factorizable imputation and conditionally complete imputation are applicable. 

Factorizable imputation can be extended to scenarios where the joint distribution of a subset of variables is identifiable by other means. We define the second method, \emph{conditionally factorizable imputation},  as follows.

\begin{definition} \label{def:conditionally_factorizable_imputation}
Let $\sM_\Omega$ be a nonparametric missing data model and let $(\+A,\+B)$ be a partition of $\+ X$. Suppose the variables in $\+ B$ are ordered as $B_{(1)} \prec \cdots \prec B_{(k)}$. A \emph{conditionally factorizable imputation method} for the partition $(\+ A, \+ B)$ consists of:
\begin{enumerate}
  \item restriction to the subpopulation with $R_{\+ A} = 1$;
  \item a collection of stochastic kernels $\xi^{\mathrm{CF}}_{\+ r}(\+ t_{\+ B} \cond \+ a, \+ h_{\+ B})$ defined for missingness patterns $\+ r$ satisfying $\+ r_{\+ A} = 1$, where
  \[
  \xi^{\mathrm{CF}}_{\+ r}(\+ t_{\+ B} \cond \+ a, \+ h_{\+ B}) =
    \prod_{j = m_{\+ B}(\+ r)}^k p(B_{(j)} \cond \+ A = \+ a, \+ B_{\prec(j)}),
  \]
  and $m_{\+ B}(\+ r)$, $\+ T_B(\+r)$, $\+ H_{\+ B}$, $\+ t_{\+ B}$ and $\+ h_{\+ B}$ are defined analogously to Definition~\ref{def:factorizable_imputation};
  \item a weight function
  \[
    w(\+ a) = c \; p(R_{\+ A} = 1 \cond \+ A = \+ a)^{-1},
  \]
  with an arbitrary normalizing constant $c$.
\end{enumerate}
\end{definition}

The next theorem provides sufficient criteria for the existence of a conditionally factorizable imputation method.
\begin{theorem} \label{thm:factorizable_extended}
The target law is identifiable and there exists a conditionally factorizable imputation method for a partition $(\+A,\+B)$ of $\+ X$ if the following conditions hold
\begin{enumerate}[label=(\alph*)]
\item $\+B \indep R_{\+ A} \cond \+A$, 
\item $p(R_{\+A}=1 \cond \+A)$ is identifiable from $p(\+A^*, R_{\+A})$, and
\item there exists an ordering $B_{(1)} \prec B_{(2)} \prec \cdots \prec B_{(k)}$ of $\+B$ for which $B_{(1)} \indep R_{B_{(1)}} \cond \+A, R_{\+A}$ 
and $B_{(j)} \indep R_{B_{(j)}},  R_{ \+ B_{\prec (j)}} \cond \+ B_{\prec (j)},\+A, R_{\+A}$,
\end{enumerate}
\end{theorem}
\begin{proof}
The target law factorizes as $p(\+X) = p(\+A)p(\+B \cond \+A)$. The identifiability of $p(\+A)$ follows directly from condition (b), which assumes the identifiability of $p(R_{\+A}=1 \cond \+A)$, because
\[
p(\+A) = \frac{p(R_{\+ A}=1)p(\+A \cond R_{\+A}=1)}{p(R_{\+A}=1 \cond \+A)}.
\]
In other words, complete cases $p(\+A \cond R_{\+A}=1)$ are weighted by $p(R_{\+ A} = 1)p(R_{\+A}=1 \cond \+A)^{-1}$.
The directly factorizable imputation is applied for the data where $\+A$ is fully observed. The validity of the imputation follows from Theorem~\ref{thm:factorizable_valid} and condition (c) when we consider an ordering where $\+A$ precedes variables $B_{(1)} \prec B_{(2)} \prec \cdots \prec B_{(k)}$. The imputations are from distribution $p(\+B \cond \+A, R_{\+A}=1)$ which by condition (a) equals $p(\+B \cond \+A)$. Thus, both components of the factorization $p(\+A)p(\+B \cond \+A)$ are identifiable.
\end{proof}

Theorem~\ref{thm:factorizable_valid} is a special case of Theorem~\ref{thm:factorizable_extended} when $\+A=\emptyset$. 
Theorem~\ref{thm:factorizable_extended} requires $p(R_{\+A}=1 \cond \+A)$ to be identifiable from $p(\+A^*, R_{\+A})$ but does not specify the way how the identification should be achieved. A practical approach will be available if $p(\+A,R_{\+A})$ is identifiable from $p(\+A^*,R_{\+A})$. Using reasoning analogous to Theorem~\ref{thm:fulllawmi}, we may conclude that conditionally complete imputation can be applied to variables $\+A$. Then the weights $p(R_{\+A}=1 \cond \+A)^{-1}$ can be estimated from the imputed data for $p(\+A,R_{\+A})$. 

Sections~\ref{exsec:conditionallyfactorizable} and \ref{exsec:anotherconditionallyfactorizable} provide examples where conditionally factorizable imputation can be applied although the full law is not identifiable and a directly factorizable imputation method does not exist. A simulation study presented in Section~\ref{sec:fourvariablesimulation} demonstrates that conditionally factorizable imputation provides unbiased estimates in these examples.

Conditionally factorizable imputation may be applicable in some longitudinal settings where $\+A$ represents baseline measurements and $\+B$ represents repeated measurements in later time points. 
The missingness of $\+B$ may be affected by $\+A$ and $R_{\+A}$. However, it is natural to assume that the missingness of $\+A$ does not depend on future measurements $\+B$, which fulfills the condition~(a) of Theorem~\ref{thm:factorizable_extended}. 
If in addition, $p(R_{\+A}=1 \cond \+A)$ can be identified from the baseline data and the repeated measurements $\+B$ ordered in time fulfill condition~(c) of Theorem~\ref{thm:factorizable_extended}, a conditionally factorizable imputation method can be applied.

\section{Practical procedure} \label{sec:procedure}

In this section, we show how the theoretical results of Section~\ref{sec:mainresult} can guide practical data analysis. As a starting point, we assume that a graphical missing data model and a random sample from the distribution $p(\+X^\ast,\+R)$ are available. Although Theorems~\ref{thm:fulllawmi} and \ref{thm:factorizable_extended} are independent of each other, it seems reasonable to first check the identifiability of the full law. The overall procedure can be presented as follows:
\begin{enumerate}
\item If the full law is identifiable, apply conditionally complete multiple imputation. 
\item If the conditions of Theorem~\ref{thm:factorizable_valid} hold for some ordering, apply factorizable imputation with this ordering. 
\item If the conditions of Theorem~\ref{thm:factorizable_extended} hold for some partition and some ordering, apply factorizable imputation with this partition and this ordering. 
\item Use the Do-search algorithm \citep{dosearch} or apply do-calculus manually to check if the target law is identifiable. If yes, estimate the target law using the obtained identifying functional. 
\item If the target law is not identifiable, it can be still checked if some marginal distribution or a joint distribution of a subset of variables can be identified. Respecify the problem and start again from Step~1.
\end{enumerate}
The above procedure is also illustrated as a flowchart in Figure~\ref{fig:flowchart}.

\begin{figure}
  \centering
  \begin{tikzpicture}[ 
      node distance=1.7cm, 
      box/.style={ 
          rectangle, 
          rounded corners, 
          draw, 
          align=center, 
          minimum width=5.5cm, 
          minimum height=0.9cm,
      }, 
      decision/.style={ 
          diamond, 
          draw, 
          align=center, 
          aspect=2, 
          inner sep=3pt 
      }, 
      line/.style={->, thick} 
  ]


  \node[box] (start) {Start:\\ Graphical missing-data model\\  and observed data};
  \node[decision, below of=start, yshift=-1.3cm] (full) {Full law\\ identifiable?}; 
  \node[box, right=1.6cm and 3.3cm of full] (ccmi) {Apply conditionally complete\\ multiple imputation};
  \node[decision, below of=full, yshift=-2cm] (thm26) {Conditions of\\ Theorem 2.6 hold?}; 
  \node[box, right=1.6cm and 3.3cm of thm26] (fmi) {Apply directly\\ factorizable imputation\\ (choose ordering)}; 
  \node[decision, below of=thm26, yshift=-2cm] (thm29) {Conditions of\\ Theorem 2.9 hold?}; 
  \node[box, right=1.6cm and 3.3cm of thm29] (cfmi) {Apply conditionally\\ factorizable imputation\\ (partition + ordering + weights)}; 
  \node[decision, below of=thm29, yshift=-2cm] (target) {Target law\\ identifiable?}; 
  \node[box, right=1.6cm and 3.3cm of target] (identify) {Use Do-search or\\ manual do-calculus\\ to derive identifying functional}; 
  \node[box, below of=target, yshift=-2cm] (subset) {Check identifiable marginals/subsets,\\ reformulate problem,\\ return to Start}; 


  \draw[line] (start) -- (full); 
  \draw[line] (full) -- node[above left]{Yes} (ccmi); 
  \draw[line] (full) -- node[right]{No} (thm26); 
  \draw[line] (thm26) -- node[above left]{Yes} (fmi); 
  \draw[line] (thm26) -- node[right]{No} (thm29); 
  \draw[line] (thm29) -- node[above left]{Yes} (cfmi); 
  \draw[line] (thm29) -- node[right]{No} (target); 
  \draw[line] (target) -- node[above left]{Yes} (identify); 
  \draw[line] (target) -- node[right]{No} (subset); 

  \end{tikzpicture}
  \caption{Flowchart depicting the practical procedure.} 
  \label{fig:flowchart}
\end{figure}

Next, we will discuss each step of the procedure. Checking the identifiability of the full law in Step~1 requires examining relatively simple graphical conditions. When unmeasured confounders are not present, the full law is identifiable if and only if the graph does not contain self-censoring edges or  colluders \citep[Theorem~1]{pmlr-v119-nabi20a}. A colluder is a specific type of collider structure, where a partially observed variable and its response indicators have an outgoing edge into the response indicator of another partially observed variable, e.g., $X \rightarrow R_Y \leftarrow R_X$. A self-censoring edge is an edge from a partially observed variable into its corresponding response indicator, e.g., $X \rightarrow R_X$. With unmeasured confounders, the full law is identifiable if and only  if the graph does not contain colluding paths \citep[Theorem~4]{pmlr-v119-nabi20a}.

The implementation of conditionally complete imputation requires some attention. Theorem~\ref{thm:fulllawmi} is an existence result: while it implies that the imputation kernels can be in principle derived from the identifying functional of the full law, it does not provide a way for obtaining a finite-sample estimators for them. In other words, the user faces the same modeling choices that are common for multiple imputation in general. The user has to choose the imputation method (linear model, predictive mean matching, random forest, etc.), the number of imputations and other details. Multiple imputation by chained equations also requires specifying the number of iterations needed to guarantee the convergence of the parameters of the imputation models.

Nevertheless, some guidance can be given on the implementation of conditionally complete imputation.
It can be seen from equation~\eqref{eq:conditionalpredictive} that the imputation model should include the response indicators \citep{tompsett2018ontheuse,beesley2021multiple}. In software packages for multiple imputation, response indicators are usually not included by default and must be manually specified in the imputation model. 

In Steps~2 and 3, the key requirement is finding a suitable ordering of the variables. While this is straightforward in a small graph, it can become challenging when dealing with a large graph. Factorizable imputation needs a preprocessing phase where the value of a variable is set as missing if any of its preceding variables has a missing value. Due to the preprocessing and the specified ordering, the covariates of the imputation models are always available and no iterations are needed. Response indicators are not used in the imputation models in factorizable imputation. Step~3 involves weights that must be estimated from the data. Standard techniques for weight / propensity score estimation are  applicable here.

Step~4 emphasizes that in some cases the target law is identifiable but neither conditionally complete multiple imputation nor factorizable imputation is applicable. In such cases, the estimation method needs to be tailored on a case-by-case basis. Finally, Step~5 remarks that even if the target law was not identified in the previous steps, it might be possible to identify the distribution of some subset of variables. The whole procedure has non-polynomial worst-case complexity with respect to the number of variables because steps~2 and 3 require a search over orderings and partitions and step 4 applies do-calculus. 

\section{Examples} \label{sec:examples}

To illustrate the application of the theoretical results of Section~\ref{sec:mainresult}, we consider multiple imputation when the data are MNAR. We investigate the identifiability of the full law and the target law in our example cases and use Theorems~\ref{thm:fulllawmi}, \ref{thm:factorizable_valid} and \ref{thm:factorizable_extended} to make conclusions about the applicability of multiple imputation. These conclusions are verified by simulation studies where the aim is to compare the methods for the estimation of basic summary statistics. In all simulation examples, we assume a multivariate normal distribution for $p(\+X)$ and logistic regression models for the response indicators. The codes for the simulation examples are available in the supplementary material.

The estimation methods in the comparison include three versions of multiple imputation, complete case analysis (CCA), available case analysis (ACA) and estimation using the identifying functionals. The imputation methods are multiple imputation without response indicators (MI), multiple imputation with response indicators (MIRI) and factorizable multiple imputation (FMI) (either directly factorizable imputation or conditionally factorizable imputation depending on the example). In FMI, some observations are replaced by missing values before the imputation to transform the data into the required monotonic missing data pattern. In all imputation methods, we apply multiple imputation by chained equations using the R package \texttt{mice} \citep{vanbuuren2011mice} with Bayesian linear regression (method ''norm``). This choice makes it possible  to closely approximate the true imputation kernels under the normal-logistic data generating mechanism used in the examples. The imputation models are specified manually for MIRI and FMI. The number of iterations for the chained equations is set to 20 in the examples with four variables. Otherwise, the default parameter settings in \texttt{mice} are used for all methods.  

\subsection{Examples with two variables} \label{sec:twovariableexamples}
We consider six graphs depicted in Figure~\ref{fig:2vars}. All graphs represent MNAR scenarios because the missingness of $Y$ depends on $X$ that can be missing as well. We study the identification and estimation of the full law $p(X,Y,R_X,R_Y)$ and the target law $p(X,Y)$ in these scenarios. 

\begin{figure}[ht]
  \centering
  \begin{subfigure}{0.3\columnwidth}
    \centering
    \begin{tikzpicture}[scale=1.5]
    \node [obs = {X}{X}] at (0,1) {};
    \node [obs = {Y}{Y}] at (1,1) {};
    \node [obs1 = {R_X}] at (0,0) {};
    \node [obs1 = {R_Y}] at (1,0) {};
    \draw[->] (X) -- (Y);
    \draw[->] (X) -- (R_Y);
    \end{tikzpicture}
  \caption{} 
  \label{fig:nocolluder}
  \end{subfigure}
  \begin{subfigure}{0.3\columnwidth}
    \centering
    \begin{tikzpicture}[scale=1.5]
    \node [obs = {X}{X}] at (0,1) {};
    \node [obs = {Y}{Y}] at (1,1) {};
    \node [obs1 = {R_X}] at (0,0) {};
    \node [obs1 = {R_Y}] at (1,0) {};
    \draw[->] (X) -- (Y);
    \draw[->] (X) -- (R_Y);
    \draw[->] (R_Y) -- (R_X);
    \end{tikzpicture}
  \caption{} 
  \label{fig:ancestorX}
  \end{subfigure}
  \begin{subfigure}{0.3\columnwidth}
    \centering
    \begin{tikzpicture}[scale=1.5]
    \node [obs = {X}{X}] at (0,1) {};
    \node [obs = {Y}{Y}] at (1,1) {};
    \node [obs1 = {R_X}] at (0,0) {};
    \node [obs1 = {R_Y}] at (1,0) {};
    \draw[->] (X) -- (Y);
    \draw[->] (X) -- (R_Y);
	  \draw[->] (Y) -- (R_X);
    \end{tikzpicture}
  \caption{} 
  \label{fig:ORcase}
  \end{subfigure}

  \begin{subfigure}{0.3\columnwidth}
    \centering
    \begin{tikzpicture}[scale=1.5]
    \node [obs = {X}{X}] at (0,1) {};
    \node [obs = {Y}{Y}] at (1,1) {};
    \node [obs1 = {R_X}] at (0,0) {};
    \node [obs1 = {R_Y}] at (1,0) {};
    \draw[->] (X) -- (Y);
    \draw[->] (X) -- (R_Y);
    \draw[->] (R_X) -- (R_Y);
    \end{tikzpicture}
  \caption{} 
  \label{fig:colluder}
  \end{subfigure}
  \begin{subfigure}{0.3\columnwidth}
    \centering
    \begin{tikzpicture}[scale=1.5]
    \node [obs = {X}{X}] at (0,1) {};
    \node [obs = {Y}{Y}] at (1,1) {};
    \node [obs1 = {R_X}] at (0,0) {};
    \node [obs1 = {R_Y}] at (1,0) {};
    \draw[->] (X) -- (Y);
    \draw[->] (X) -- (R_Y);
	  \draw[->] (Y) -- (R_X);
	  \draw[->] (R_X) -- (R_Y);
    \end{tikzpicture}
  \caption{} 
  \label{fig:nonid}
  \end{subfigure}
  \begin{subfigure}{0.3\columnwidth}
    \centering
    \begin{tikzpicture}[scale=1.5]
    \node [obs = {X}{X}] at (0,1) {};
    \node [obs = {Y}{Y}] at (1,1) {};
    \node [obs1 = {R_X}] at (0,0) {};
    \node [obs1 = {R_Y}] at (1,0) {};
    \draw[->] (X) -- (R_Y);
	  \draw[->] (Y) -- (R_X);
	  \draw[->] (R_X) -- (R_Y);
    \end{tikzpicture}
  \caption{} 
  \label{fig:XYindependent}
  \end{subfigure}
  \caption{Graphs used for the two-variable examples. The full law and the target law are identifiable in graphs (a), (b) and (c). The target law is identifiable also in graphs (d) and (f) but not in graph (e). Proxy variables and edges related to them are not shown for simplicity.} 
  \label{fig:2vars}
\end{figure}

First, we consider the identifiability of the full law $p(X,Y,R_X,R_Y)$ and the target law $p(X,Y)$ graph by graph. The identifying functionals can be derived manually or by applying the Do-search algorithm, for example. The results are summarized in Table~\ref{tab:idresults}.

\begin{table}[!htb]
  \centering
  \caption{Summary of identifiability and validity of imputation for the graphs of Figure~\ref{fig:2vars}.} \label{tab:idresults}
  \begin{tabular}{lcccccc}
    \toprule
    & \multicolumn{6}{c}{Graph} \\
    & 1a & 1b & 1c & 1d & 1e & 1f \\
    \midrule
    Full law identifiable & \checkmark & \checkmark & \checkmark & $\times$ & $\times$ & $\times$ \\
    Target law identifiable & \checkmark & \checkmark & \checkmark & \checkmark & $\times$ & \checkmark \\
    Theorem~\ref{thm:fulllawmi} applies & \checkmark & \checkmark & \checkmark & $\times$ & $\times$ & $\times$ \\
    Theorem~\ref{thm:factorizable_valid} applies & \checkmark & $\times$ & $\times$ & \checkmark & $\times$ & $\times$ \\
    \bottomrule
  \end{tabular}
\end{table}

\subsubsection{Simple MNAR example} \label{exsec:simplemnar}
In the graph of Figure~\ref{fig:nocolluder}, the full law is identified by the formula
\[
  p(X,Y,R_X,R_Y)  
    = p(Y^\ast \cond X^\ast,R_Y=1,R_X=1)p(R_Y \cond X^\ast,R_X=1) 
    p(X^\ast \cond R_X=1)p(R_X).
\]
thus a conditionally complete imputation method exists. The conditional distributions for imputation are identified using the full law as follows
\begin{equation} \label{eq:exampleconditional}
  \begin{aligned}
  p(X \cond Y,R_X = 0,R_Y = 1) &= \frac{p(X,Y,R_X = 0,R_Y = 1)}{\sum\limits_X p(X,Y,R_X =  0,R_Y = 1)}, \\
  p(Y \cond X,R_X = 1,R_Y = 0) &= \frac{p(Y,X,R_X = 1,R_Y = 0)}{\sum\limits_Y p(Y,X,R_X = 1,R_Y = 0)}, \\
  p(X,Y \cond R_X = 0,R_Y = 0) &= \frac{p(Y,X,R_X = 0,R_Y = 0)}{\sum\limits_{X,Y} p(Y,X,R_X = 0,R_Y = 0)},
  \end{aligned}
\end{equation}
where the summations should be understood as integrals when $X$ and $Y$ are continuous variables. The target law is identified by the formula
\begin{equation} \label{eq:targetlaw_nocolluder}
  p(X,Y) = p(Y^\ast \cond X^\ast, R_X = 1, R_Y = 1) p(X^\ast \cond R_X = 1),
\end{equation}
which can be seen by noting that $Y \indep R_Y \cond X$ and $X \indep R_X$. The conditions of Theorem~\ref{thm:factorizable_valid} are thus satisfied by the ordering $X \prec Y$ which implies that a factorizable imputation method also exists. Applying Corollary~2 by \citet{Mohan_2021}, we conclude that neither CCA nor ACA is a valid estimation method for the target law. We also see that the m-backdoor criterion \citep{mathur2024imputation} is not applicable in this case. These conclusions on the applicability of CCA, ACA and the m-backdoor criterion apply also to the other graphs in Figure~\ref{fig:2vars}.

\subsubsection{Mediated self-censoring example} \label{exsec:mediatedselfcensoring}
The graph of Figure~\ref{fig:ancestorX} has an additional edge $R_Y \rightarrow R_X$ compared to Figure~\ref{fig:nocolluder}. The full law is still identifiable and the identifying functional is
\begin{equation}
  \begin{aligned}
    &p(X,Y,R_X,R_Y) 
    = p(Y^\ast \cond X^\ast,R_X = 1,R_Y = 1)p(X^\ast,R_X = 1) 
    \frac{p(R_Y \cond X^\ast,R_X = 1) p(R_X|R_Y)}{p(R_X = 1 \cond R_Y)}.
  \end{aligned}
\end{equation}
Therefore, a conditionally complete imputation method exists also in this case. The target law is identified by the formula
\begin{equation} \label{eq:targetlaw_ancestorX}
  \begin{aligned}
    p(X,Y) &= p(Y^\ast \cond X^\ast, R_X = 1, R_Y = 1) 
    \sum_{R_Y,Y^\ast} p(R_Y,Y^\ast)p(X^\ast \cond R_X=1,R_Y,Y^\ast). 
  \end{aligned}
\end{equation}
As neither $X$ nor $Y$ is independent from its response indicator, the conditions of Theorem~\ref{thm:factorizable_valid} cannot be satisfied.

\subsubsection{Identification via odds ratio example} \label{exsec:or}
In the graph of Figure~\ref{fig:ORcase}, the full law is identifiable via the odds ratio factorization \citep{Yun_Chen_2007} which can be used to first identify the missingness mechanism as
\begin{equation} \label{eq:or_factorization}
  p(R_X, R_Y \cond X, Y) = \frac{\alpha(R_X,R_Y,X,Y)}{\beta(X,Y)}
\end{equation}
where
\[
  \begin{aligned}
    \alpha(R_X,R_Y,X,Y) &= p(R_X \cond R_Y \eqs 1, X, Y) p(R_Y \cond R_X \eqs 1, X, Y) 
    \mathrm{OR}(R_Y, R_X \cond X, Y), \\
    \beta(X,Y) &= \sum_{r_x,r_y} \alpha(R_X = r_x, R_Y = r_y, X, Y), \\
    \mathrm{OR}(R_Y, R_X \cond X, Y) &= \frac{p(R_Y \cond R_X, X, Y)}{p(R_Y = 1 \cond R_X, X, Y)} \frac{p(R_Y = 1 \cond R_X = 1, X, Y)}{p(R_Y \cond R_X = 1, X, Y)}.
  \end{aligned}
\]
All terms on the right-hand side of equation~\eqref{eq:or_factorization} are identifiable as no colluders or self-censoring edges are present. This implies the identifiability of the full law and consequently the existence of a conditionally complete imputation method. The identifying functional for the target law is
\begin{equation} \label{eq:targetlaw_ORcase}
  \begin{aligned}
    p(X,Y) &= \frac{p(R_X=1) p(X^\ast \cond R_X = 1)}{p(R_X = 1 \cond Y^\ast,R_Y = 1)} 
    p(Y^\ast \cond X^\ast,R_X = 1,R_Y = 1).
  \end{aligned}
\end{equation}
Like the previous graph, neither $X$ nor $Y$ is independent from its response indicator and the conditions of Theorem~\ref{thm:factorizable_valid} cannot be satisfied.

\subsubsection{Colluder example} \label{exsec:colluder}
The graph of Figure~\ref{fig:colluder} has a colluder structure $X \rightarrow R_Y \leftarrow R_X$. We conclude that the full law is not identifiable and consequently a conditionally complete imputation method does not exist. The identifying formula for the target law is
\begin{equation} \label{eq:targetlaw_colluder}
  p(X,Y) = p(Y^\ast \cond X^\ast, R_X = 1, R_Y = 1)p(X^\ast \cond R_X = 1). 
\end{equation}
We have that $Y \indep R_Y \cond X$ and $X \indep R_X$ implying that a factorizable imputation method can be used under the ordering $X \prec Y$.

\subsubsection{Criss-cross example} \label{exsec:crisscross}
The graph of Figure~\ref{fig:nonid} exemplifies a case where neither the full law nor the target law can be identified. The non-identifiability of the full law follows again from the presence of the colluder structure $R_X \rightarrow R_Y \leftarrow X$. The non-identifiability of the target law was demonstrated in this case by \citet{nabi2023testability}. Therefore, neither conditionally complete imputation nor factorizable imputation can be applied.

\subsubsection{Independent variables example} \label{exsec:independent}
In the graph of Figure~\ref{fig:XYindependent}, $X$ and $Y$ are independent. The full law cannot be identified because of the colluder $X \rightarrow R_Y \leftarrow R_X$ but the target law is identified by the formula
\begin{equation} \label{eq:targetlaw_XYindependent}
  p(X,Y) = p(X^\ast \cond R_X=1) \sum_{R_X} p(R_X)(Y^\ast \cond R_Y = 1, R_X). 
\end{equation}
In this case, $X$ and $R_X$ are independent but $Y$ is not independent from $R_X$ and $R_Y$. Thus, the conditions of Theorem~\ref{thm:factorizable_valid} are not fulfilled. 

\subsubsection{Simulation study for the two-variable examples} \label{sec:twovariablesimulation}
Next, we will focus on the large-sample estimation of $p(X,Y)$. For this purpose, we simulate observations from joint distributions $p(X^\ast,Y^\ast,R_X,R_Y)$ that factorize according to the graphs of Figure~\ref{fig:2vars}. The simulation examples are referred to as Example~1a, Example~1b, etc. We simulate data on $X$ and $Y$ from a bivariate normal distribution where the expectations are $0$, the variance of $X$ is $1$, the variance of $Y$ is $2$, and the correlation between $X$ and $Y$ is $\sqrt{2}/2 \approx 0.71$ with the exception of Example~\ref{fig:XYindependent} where the correlation is zero. The response indicator $R_X$ is simulated from a Bernoulli distribution where $p(R_X = 1) = 0.7$ in Examples~\ref{fig:nocolluder} and \ref{fig:colluder}, $p(R_X = 1) = 0.4 + 0.3 R_Y$ in Example~\ref{fig:ancestorX}, and $\textrm{logit}(p(R_Y = 1)) = X$ in Examples~\ref{fig:ORcase}, \ref{fig:nonid} and \ref{fig:XYindependent}. The response indicator $R_Y$ is simulated from a Bernoulli distribution where $\textrm{logit}(p(R_Y = 1)) = X$ in Examples~\ref{fig:nocolluder}, \ref{fig:ancestorX} and \ref{fig:ORcase} and from a Bernoulli distribution where $\textrm{logit}(p(R_Y=1)) = (X + 0.4) (2 R_X - 1)$ in Examples~\ref{fig:colluder}, \ref{fig:nonid} and \ref{fig:XYindependent}. In the latter case, the functional form is chosen so that $X$ and $R_X$ have a strong interaction when the graph has the colluder $X \rightarrow R_Y \leftarrow R_X$. As we are interested in the validity of imputation, we consider only large-sample properties and set the sample size to one million with no repetitions because the sample variation is negligible. The observed data are a random sample from $p(X^\ast,Y^\ast,R_X,R_Y)$. 

The simulation results in Table~\ref{tab:simulation_results} are in agreement with the theoretical considerations. MIRI provides unbiased estimates in Examples~\ref{fig:nocolluder}, \ref{fig:ancestorX} and \ref{fig:ORcase} as expected and suffers from some bias in other examples. MI is always biased because the data are MNAR and therefore the imputation models cannot be correctly specified without the response indicators. FMI gives unbiased estimates in Examples~\ref{fig:nocolluder} and \ref{fig:colluder} where the ordering $X \prec Y$ fulfills the conditions of Theorem~\ref{thm:factorizable_valid} because $X \indep R_X$ and $Y \indep R_Y \cond X$. As the first step of FMI, the observed values of $Y$ are set to be missing when $R_X = 0$. FMI draws imputations in the ordering $X \prec Y$ regardless of whether Theorem~\ref{thm:factorizable_valid} applies. CCA and CCA provide biased estimates for the target law in all examples. The use of identifying functionals leads to unbiased estimates in Examples~\ref{fig:nocolluder}, \ref{fig:ancestorX}, \ref{fig:colluder}, \ref{fig:ORcase} and \ref{fig:XYindependent} where the target law is identifiable. In Example~\ref{fig:nonid}, the target law is not identifiable and an identifying functional does not exist. 

\begin{table*}[!htb]
\centering
\caption{Simulation results for the examples illustrated in Figure~\ref{fig:2vars}. The biases are reported for multiple imputation without response indicators (MI), multiple imputation with response indicators (MIRI), factorizable multiple imputation (FMI), complete case analysis (CCA), available case analysis (ACA), and estimation using the identifying functional (IF). 
} \label{tab:simulation_results}
\begin{tabular}{clrrrrrrr}
\toprule
& & & \multicolumn{6}{c}{Bias} \\
Example & Statistic & Truth & MI & MIRI & FMI & CCA & ACA & IF  \\
\midrule
\multirow{4}{*}{1a} &  $\E(X)$  &  $0.00$  &  $0.03$  &  $0.00$  &  $0.00$  &  $0.41$  &  $0.00$  &  $0.00$ \\ 
 & $\E(Y)$  &  $0.00$  &  $0.11$  &  $0.00$  &  $0.00$  &  $0.41$  &  $0.41$  &  $0.00$ \\ 
 & $\sd(X)$  &  $1.00$  &  $0.00$  &  $0.00$  &  $0.00$  &  $-0.09$  &  $0.00$  &  $0.00$ \\ 
 & $\sd(Y)$  &  $1.41$  &  $-0.02$  &  $0.00$  &  $0.00$  &  $-0.06$  &  $-0.06$  &  $0.00$ \\ 
 & $\Cor(X,Y)$  &  $0.71$  &  $-0.01$  &  $0.00$  &  $0.00$  &  $-0.03$  &  $-0.03$  &  $0.00$ \\
\midrule
\multirow{4}{*}{1b} &  $\E(X)$  &  $0.00$  &  $0.14$  &  $0.00$  &  $0.11$  &  $0.41$  &  $0.11$  &  $0.00$ \\ 
&  $\E(Y)$  &  $0.00$  &  $0.20$  &  $0.00$  &  $0.11$  &  $0.41$  &  $0.41$  &  $0.00$ \\ 
& $\sd(X)$  &  $1.00$  &  $-0.01$  &  $0.00$  &  $-0.01$  &  $-0.09$  &  $-0.01$  &  $0.00$ \\ 
&  $\sd(Y)$  &  $1.41$  &  $-0.02$  &  $0.00$  &  $0.00$  &  $-0.06$  &  $-0.06$  &  $0.00$ \\ 
&  $\Cor(X,Y)$  &  $0.71$  &  $-0.01$  &  $0.00$  &  $0.00$  &  $-0.03$  &  $-0.03$  &  $0.00$ \\
\midrule
\multirow{4}{*}{1c} &  $\E(X)$  &  $0.00$  &  $0.22$  &  $0.00$  &  $0.36$  &  $0.67$  &  $0.36$  &  $0.00$ \\ 
&  $\E(Y)$  &  $0.00$  &  $0.29$  &  $0.00$  &  $0.73$  &  $0.99$  &  $0.41$  &  $0.00$ \\ 
&  $\sd(X)$  &  $1.00$  &  $-0.06$  &  $0.00$  &  $-0.07$  &  $-0.14$  &  $-0.07$  &  $-0.01$ \\ 
&  $\sd(Y)$  &  $1.41$  &  $-0.07$  &  $0.00$  &  $-0.19$  &  $-0.23$  &  $-0.06$  &  $0.00$ \\ 
&  $\Cor(X,Y)$  &  $0.71$  &  $-0.01$  &  $0.00$  &  $-0.05$  &  $-0.08$  &  $-0.08$  &  $-0.01$ \\
\midrule
\multirow{4}{*}{1d} &  $\E(X)$  &  $0.00$  &  $-0.03$  &  $-0.15$  &  $0.00$  &  $0.35$  &  $0.00$  &  $0.00$ \\ 
 & $\E(Y)$  &  $0.00$  &  $-0.10$  &  $-0.29$  &  $0.00$  &  $0.35$  &  $0.15$  &  $0.00$ \\ 
 & $\sd(X)$  &  $1.00$  &  $0.00$  &  $0.03$  &  $0.00$  &  $-0.08$  &  $0.00$  &  $0.00$ \\ 
 & $\sd(Y)$  &  $1.41$  &  $0.01$  &  $0.07$  &  $0.00$  &  $-0.06$  &  $-0.01$  &  $0.00$ \\ 
 & $\Cor(X,Y)$  &  $0.71$  &  $0.01$  &  $0.02$  &  $0.00$  &  $-0.03$  &  $-0.03$  &  $0.00$ \\
\midrule
\multirow{4}{*}{1e} &  $\E(X)$  &  $0.00$  &  $0.12$  &  $-0.15$  &  $0.36$  &  $0.62$  &  $0.36$  &  NA \\ 
&  $\E(Y)$  &  $0.00$  &  $0.01$  &  $-0.32$  &  $0.73$  &  $0.94$  &  $0.09$  &  NA \\ 
&  $\sd(X)$  &  $1.00$  &  $-0.03$  &  $0.06$  &  $-0.07$  &  $-0.13$  &  $-0.07$  &  NA \\ 
&  $\sd(Y)$  &  $1.41$  &  $0.10$  &  $0.19$  &  $-0.20$  &  $-0.23$  &  $0.12$  &  NA \\ 
&  $\Cor(X,Y)$  &  $0.71$  &  $0.02$  &  $0.04$  &  $-0.06$  &  $-0.08$  &  $-0.08$  &  NA \\
\midrule
\multirow{4}{*}{1f} &  $\E(X)$  &  $0.00$  &  $-0.04$  &  $-0.07$  &  $0.00$  &  $0.35$  &  $0.00$  &  $0.00$ \\ 
&  $\E(Y)$  &  $0.00$  &  $0.07$  &  $0.00$  &  $0.72$  &  $0.72$  &  $0.12$  &  $0.00$ \\ 
&  $\sd(X)$  &  $1.00$  &  $0.00$  &  $0.00$  &  $0.00$  &  $-0.09$  &  $0.00$  &  $0.00$ \\ 
&  $\sd(Y)$  &  $1.41$  &  $0.00$  &  $0.00$  &  $-0.20$  &  $-0.20$  &  $0.00$  &  $0.00$ \\ 
&  $\Cor(X,Y)$  &  $0.00$  &  $0.19$  &  $0.04$  &  $0.00$  &  $0.00$  &  $0.00$  &  $0.00$ \\
\bottomrule
\end{tabular}
\vspace*{2cm}
\end{table*}

\subsection{Examples with four variables} \label{sec:fourvariableexamples}
Figure~\ref{fig:4vars} presents four graphs with four variables. The full law $p(X,W,Z,Y,R_X,R_W,R_Z,R_Y)$ is not identifiable in these graphs due to the colluders: $Z \rightarrow R_W \leftarrow R_Z$ in graphs~\ref{fig:4varsfactorizes} and \ref{fig:4varsdonotfactorize}, $W \rightarrow R_Z \leftarrow R_W$ and $Y \rightarrow R_Z \leftarrow R_Y$ in graph~\ref{fig:4varsweightingmi1}, and  $W \rightarrow R_X \leftarrow R_W$ and $Z \rightarrow R_Y \leftarrow R_Z$ in graph~\ref{fig:4varsweightingmi2}. Using Theorem~\ref{thm:fulllawmi}, we conclude that a conditionally complete imputation method does not exist. The target law $p(X,W,Z,Y)$ is identifiable in all graphs of Figure~\ref{fig:4vars} as we show below.

\begin{figure}[ht]
\centering
\begin{subfigure}{0.48\columnwidth}  
  \begin{center}
    \begin{tikzpicture}[scale=1.8]
    \node [obs = {X}{X}] at (0,1) {};
    \node [obs = {W}{W}] at (1,1) {};
    \node [obs = {Z}{Z}] at (2,1) {};
    \node [obs = {Y}{Y}] at (3,1) {};
    \node [obs1 = {R_X}] at (0,0) {};
    \node [obs1 = {R_W}] at (1,0) {};
    \node [obs1 = {R_Z}] at (2,0) {};
    \node [obs1 = {R_Y}] at (3,0) {};
    \draw[->, bend left=35] (X) to (Y);
    \draw[->] (X) -- (W);
    \draw[->] (X) -- (R_Y);
    \draw[->] (W) -- (Z);
    \draw[->] (W) -- (R_X);
    \draw[->] (Z) -- (Y);
    \draw[->] (Z) -- (R_W);
    \draw[->] (R_Z) -- (R_W);
    	\draw[->, bend right=35, draw=none] (R_X) to (R_Y); 
    \end{tikzpicture}
  \end{center}
  \caption{}
  \label{fig:4varsfactorizes}
  \end{subfigure}
  \begin{subfigure}{0.48\columnwidth}  
  \begin{center}
    \begin{tikzpicture}[scale=1.8]
    \node [obs = {X}{X}] at (0,1) {};
    \node [obs = {W}{W}] at (1,1) {};
    \node [obs = {Z}{Z}] at (2,1) {};
    \node [obs = {Y}{Y}] at (3,1) {};
    \node [obs1 = {R_X}] at (0,0) {};
    \node [obs1 = {R_W}] at (1,0) {};
    \node [obs1 = {R_Z}] at (2,0) {};
    \node [obs1 = {R_Y}] at (3,0) {};
    \draw[->] (X) -- (W);
    \draw[->, bend left=25] (X) to (Z);
    \draw[->, bend left=35] (X) to (Y);
	\draw[->] (W) -- (Z);
	\draw[->, bend left=25] (W) to (Y);
	\draw[->] (Z) -- (Y);
    \draw[->] (X) -- (R_Y);
    \draw[->] (X) -- (R_W);
    \draw[->] (W) -- (R_X);
    \draw[->] (W) -- (R_Z);
    \draw[->] (Y) -- (R_Z);
	\draw[->, bend right=25] (R_X) to (R_Z);  
	\draw[->, bend right=35] (R_X) to (R_Y);    
    \draw[->] (R_W) -- (R_Z);
    \draw[->] (R_Y) -- (R_Z);
    \end{tikzpicture}
  \end{center}
  \caption{}
  \label{fig:4varsweightingmi1}
  \end{subfigure}

  \begin{subfigure}{0.48\columnwidth}  
  \begin{center}
    \begin{tikzpicture}[scale=1.8]
    \node [obs = {X}{X}] at (0,1) {};
    \node [obs = {W}{W}] at (1,1) {};
    \node [obs = {Z}{Z}] at (2,1) {};
    \node [obs = {Y}{Y}] at (3,1) {};
    \node [obs1 = {R_X}] at (0,0) {};
    \node [obs1 = {R_W}] at (1,0) {};
    \node [obs1 = {R_Z}] at (2,0) {};
    \node [obs1 = {R_Y}] at (3,0) {};
    \draw[->] (X) -- (W);
    \draw[->, bend left=25] (X) to (Z);
    \draw[->, bend left=35] (X) to (Y);
	\draw[->] (W) -- (Z);
	\draw[->, bend left=25] (W) to (Y);
	\draw[->] (Z) -- (Y);
	\draw[->] (X) -- (R_Y);
    \draw[->] (W) -- (R_X);
    \draw[->] (W) -- (R_Z);
    \draw[->] (W) -- (R_Y);
    \draw[->] (Z) -- (R_X);
    \draw[->] (Z) -- (R_W);
    \draw[->] (Z) -- (R_Y);
    \draw[->, bend right=25] (R_W) to (R_Y);
    \draw[->] (R_W) -- (R_X);
    \draw[->] (R_Z) -- (R_Y);
    \draw[->, bend left=25] (R_Z) to (R_X);
    \end{tikzpicture}
  \end{center}
  \caption{}
  \label{fig:4varsweightingmi2}
  \end{subfigure}
  \begin{subfigure}{0.48\columnwidth}  
  \begin{center}
    \begin{tikzpicture}[scale=1.8]
    \node [obs = {X}{X}] at (0,1) {};
    \node [obs = {W}{W}] at (1,1) {};
    \node [obs = {Z}{Z}] at (2,1) {};
    \node [obs = {Y}{Y}] at (3,1) {};
    \node [obs1 = {R_X}] at (0,0) {};
    \node [obs1 = {R_W}] at (1,0) {};
    \node [obs1 = {R_Z}] at (2,0) {};
    \node [obs1 = {R_Y}] at (3,0) {};
    	\draw[->, bend right=45,draw=none] (R_W) to (R_Y); 
    \draw[->, bend left=35] (X) to (Y);
    \draw[->] (X) -- (W);
    \draw[->] (X) -- (R_Y);
    \draw[->] (X) -- (R_Z);
    \draw[->] (W) -- (Z);
    \draw[->] (W) -- (R_X);
    \draw[->] (Z) -- (Y);
    \draw[->] (Z) -- (R_W);
    \draw[->] (R_Z) -- (R_W);
    \end{tikzpicture}
    \caption{}
    \label{fig:4varsdonotfactorize}
  \end{center}
  \end{subfigure}
   \caption{Graphs where  the full law is not identifiable but the target law can be identified. In graph~(a), a directly factorizable imputation method can be used in the ordering $Z \prec W \prec X \prec Y$. Conditionally factorizable imputation can be applied for the partition $(\{X,W\},\{Z,Y\})$ in graph~(b), and for the partition $(\{W,Z\},\{X,Y\})$ in graph~(c). In graph~(d), the conditions of Theorem~\ref{thm:factorizable_valid} or Theorem~\ref{thm:factorizable_extended} are not satisfied but  the target law is still identifiable. Proxy variables and edges related to them are not shown for simplicity.} 
  \label{fig:4vars}
\end{figure}

\subsubsection{Directly factorizable imputation example} \label{exsec:factorizable}
In the graph of Figure~\ref{fig:4varsfactorizes} we can order the partially observed variables as $Z \prec W \prec X \prec Y$ so that the ordering satisfies the conditional independence restrictions required in Theorem~\ref{thm:factorizable_valid}, which are
\begin{align*}
Z & \indep R_Z, \\
W & \indep R_Z, R_W \cond Z, \\
X & \indep R_Z, R_W, R_X \cond Z,W, \\
Y & \indep R_Z, R_W, R_X, R_Y \cond Z,W,X. 
\end{align*}
Thus, a directly factorizable imputation method exists for the ordering $Z \prec W \prec X \prec Y$ and the imputations are drawn as follows:
\begin{align}
& \textrm{If } R_Z=0, \textrm{impute } Z \textrm{ from } p(Z \cond R_Z=1). \label{eq:impZ} \\
& \textrm{If } R_Z R_W=0, \textrm{impute } W \textrm{ from } p(W \cond Z, R_Z R_W=1). \label{eq:impW} \\
& \textrm{If } R_Z R_W R_X=0, \textrm{impute } X \textrm{ from } p(X \cond Z,W, R_Z R_W R_X=1). \label{eq:impX}\\
& \textrm{If } R_Z R_W R_X R_Y=0,  \textrm{impute } Y \textrm{ from } p(Y \cond Z,W,X, R_Z R_W R_X R_Y=1) \label{eq:impY}.
\end{align}
From the graph of Figure~\ref{fig:4varsfactorizes}, we also obtain the conditional independence restrictions $X \indep Z \cond W$ and $Y \indep W \cond X,Z$ that can be used to simplify the imputation models for $X$ and $Y$, respectively.

\subsubsection{Conditionally factorizable imputation example} \label{exsec:conditionallyfactorizable}
In the graph of Figure~\ref{fig:4varsweightingmi1}, an ordering required for the directly factorizable imputation does not exist. However, we can apply Theorem~\ref{thm:factorizable_extended} for the partition $(\{X,W\},\{Z,Y\})$. To see this, we need to check that the conditions of Theorem~\ref{thm:factorizable_extended} are fulfilled. From the graph we see that $Z,Y \indep R_X,R_W \cond X,W$ which is the requirement of condition (a) of Theorem~\ref{thm:factorizable_extended}. The identifiability of $p(R_{XW}=1 \cond X,W)$ from $p(X^*,W^*,R_X,R_W)$ can be confirmed by applying Do-search, which fulfills the condition (b) of Theorem~\ref{thm:factorizable_extended}. This implies that we can apply conditionally complete imputation for $X$ and $W$ and then use the imputed data to estimate the probabilities $p(R_{XW}=1 \cond X,W)$. Finally, condition (c) of Theorem~\ref{thm:factorizable_extended} holds for the ordering $Y \prec Z$ because  $Y \indep R_Y \cond X,W,R_X,R_W$ and $Z \indep R_Z,R_Y \cond Y,X,W,R_X,R_W$. This means that a conditionally factorizable imputation method exists and imputations are drawn for the data where $R_X R_W = 1$ from the following distributions
\begin{align}
& \textrm{If } R_Y=0, \textrm{impute } Y \textrm{ from } p(Y \cond X,W,R_X R_W R_Y=1). \label{eq:impY_2b} \\
& \textrm{If } R_Y R_Z=0, \textrm{impute } Z \textrm{ from } p(Z \cond X,W,Y, R_X R_W R_Y R_Z=1). \label{eq:impZ_2b}
\end{align}

\subsubsection{Another conditionally factorizable imputation example} \label{exsec:anotherconditionallyfactorizable}
The reasoning is similar for Figure~\ref{fig:4varsweightingmi2}. This time, Theorem~\ref{thm:factorizable_extended} applies for the partition $(\{W,Z\},\{X,Y\})$ and the ordering is $X \prec Y$. The full law $p(W,Z,R_W,R_Z)$ is identifiable from $p(W^*,Z^*,R_W,R_Z)$ licensing the use of conditionally complete imputation in the weight estimation. The conditions  $X \indep R_X \cond W,Z,R_W,R_Z$ and $Y \indep R_Y,R_X \cond X,W,Z,R_W,R_Z$ hold permitting the factorizable imputation of $X$ and $Y$ given $W$ and $Z$ when $R_{WZ}=1$. Thus, a conditionally factorizable imputation method exists and imputations are drawn for the data where $R_W R_Z = 1$ from the following distributions
\begin{align}
& \textrm{If } R_X=0, \textrm{impute } X \textrm{ from } p(X \cond W,Z,R_W R_Z R_X=1). \label{eq:impX_2c} \\
& \textrm{If } R_X R_Y=0, \textrm{impute } Y \textrm{ from } p(Y \cond W,Z,X, R_W R_Z R_X R_Y=1). \label{eq:impY_2c}
\end{align}

\subsubsection{Example where target law is identified by other means} \label{exsec:othermeans}
The graph of Figure~\ref{fig:4varsdonotfactorize} exemplifies a case where neither Theorem~\ref{thm:factorizable_valid} nor Theorem~\ref{thm:factorizable_extended} is not applicable but the target law is still identifiable. The graph is similar to the graph of Figure~\ref{fig:4varsfactorizes} but there is an additional edge $X \rightarrow R_Z$. 
In this scenario, it is possible to derive multiple identifying functionals for the target law. We consider two of such functionals. 

The first functional is obtained using Do-search and it assigns weights for the complete observations
\begin{equation} \label{eq:4varsb_allprobs}
  \begin{aligned}
  & p(X,W,Z,Y) 
  = \frac{p(X, W, Z, Y, R_X = 1, R_W = 1, R_Z = 1, R_Y = 1)}{p(R_X = 1|W,R_W = 1)p(R_W = 1|Z,R_Z = 1)p(R_Z = 1,R_Y = 1|X,R_X = 1)}. 
\end{aligned}
\end{equation}
The numerator of functional~\eqref{eq:4varsb_allprobs} is simple to estimate because it corresponds to complete observations. The denominator is defined as a product of  conditional response probabilities and may require large sample sizes for reliable estimation.

The second functional is obtained from the first one by manual derivation
\begin{equation} \label{eq:4varsb_rxprob}
  \begin{aligned}
  & p(X,W,Z,Y) 
  = \frac{p(R_X=1)p(X|R_X = 1)p(Z|X,R_X R_Z = 1)p(W,Y|X,Z,R_X R_W R_Z R_Y = 1)}{p(R_X = 1|W,R_W = 1)}. 
\end{aligned}
\end{equation}
The numerator of functional~\eqref{eq:4varsb_rxprob} consists of conditional distributions that can be estimated one by one from the specified subsets of the data. The denominator contains only one conditional response probability term and is therefore more stable to be estimated than the denominator of functional~\eqref{eq:4varsb_allprobs}. 

\subsubsection{Simulation study for the four-variable examples} \label{sec:fourvariablesimulation}
In Examples~\ref{fig:4varsfactorizes} -- \ref{fig:4varsdonotfactorize}, we simulate one million observations from a linear Gaussian model constructed according to the graphs of Figure~\ref{fig:4varsfactorizes} -- \ref{fig:4varsdonotfactorize}, respectively. The detailed description of the simulation models is presented in Appendix~\ref{apx:fourvar}. The observed data are a random sample from $p(X^\ast,W^\ast,Z^\ast,Y^\ast,R_X,R_W,R_Z,R_Y)$.

The methods in comparison are the same as in Example~\ref{fig:2vars} with the addition of conditionally factorizable imputation. MI uses the partially observed variables $X$, $W$, $Z$ and $Y$ to impute each other. MIRI uses imputation models where these variables have interactions with the response indicators. CCA and ACA are included as well.

FMI operates differently in each example. In 
Example~\ref{fig:4varsfactorizes}, FMI uses modified data where $W^\ast$ is set to be missing if $R_Z=0$, $X^\ast$ is set to be missing if $R_Z=0$ or $R_W=0$, and $Y^\ast$ is set to be missing if $R_Z=0$ or $R_W=0$ or $R_X=0$. The imputation models are specified according to equations~\eqref{eq:impZ}, \eqref{eq:impW}, \eqref{eq:impX} and \eqref{eq:impY}.

In Example~\ref{fig:4varsweightingmi1}, FMI uses first MIRI to obtain imputed data on $p(X,W,R_X,R_W)$. Then a generalized additive model is fitted to estimate probabilities $p(R_X R_W = 1 \cond X,W)$ whose inverses will be the analysis weights. The data are restricted to the subset where $R_X R_W = 1$ and $Y$ and $Z$ are imputed according to equations~\eqref{eq:impY_2b} and \eqref{eq:impZ_2b}, respectively.

In Example~\ref{fig:4varsweightingmi2}, FMI uses first MIRI to obtain imputed data on $p(W,Z,R_W,R_Z)$. Then a generalized additive model is fitted to estimate probabilities $p(R_W R_Z = 1 \cond W,Z)$ whose inverses will be the analysis weights. The data are restricted to the subset where $R_W R_Z = 1$ and $X$ and $Y$ are imputed according to equations~\eqref{eq:impX_2c} and \eqref{eq:impY_2c}, respectively.

In Example~\ref{fig:4varsdonotfactorize}, the  conditions required for FMI are not met for any partition. However, for the comparison, FMI is applied in same way as in Example~\ref{fig:4varsfactorizes}. The identifying functionals of equations~\eqref{eq:4varsb_allprobs} and \eqref{eq:4varsb_rxprob} are applied in estimation in Example~\ref{fig:4varsdonotfactorize}.

From Table~\ref{tab:simulation_resultsXWZY}, we can see that the estimates are unbiased for FMI and biased for the other methods in Examples~\ref{fig:4varsfactorizes}, \ref{fig:4varsweightingmi1} and \ref{fig:4varsweightingmi2}. In Example~\ref{fig:4varsdonotfactorize}, FMI also provides biased results as expected but estimation using equations~\eqref{eq:4varsb_allprobs} and \eqref{eq:4varsb_rxprob} gives unbiased or nearly unbiased estimates.  The same conclusions hold for standard deviations and correlations which are reported in Tables~\ref{tab:simulation_resultsXWZY_sdcor_abc} and \ref{tab:simulation_resultsXWZY_sdcor_d} in Appendix. The simulation results are thus in agreement with Theorems~\ref{thm:fulllawmi}, \ref{thm:factorizable_valid} and \ref{thm:factorizable_extended}.

\begin{table}[ht]
\centering
\caption{Simulation results for the examples illustrated in Figure~\ref{fig:4vars}. The biases are reported for factorizable multiple imputation (FMI),  multiple imputation without response indicators (MI), multiple imputation with response indicators (MIRI),  complete case analysis (CCA) and available case analysis (ACA), and in the case of Example~\ref{fig:4varsdonotfactorize}, for estimation using the referred equations (Eq.).
} \label{tab:simulation_resultsXWZY}
\begin{tabular}{clrrrrr}
\toprule
 & & \multicolumn{5}{c}{Bias} \\
Example & Statistic & FMI & MI & MIRI &  CCA & ACA \\
\midrule
\multirow{4}{*}{2a} & $\E(X)=0$  &  $0.00$  &  $0.09$  &  $0.00$  &  $0.72$  &  $0.29$ \\ 
 & $\E(W)=0$  &  $0.00$  &  $-0.01$  &  $-0.06$  &  $0.78$  &  $0.11$ \\ 
 & $\E(Z)=0$  &  $0.00$  &  $-0.01$  &  $-0.04$  &  $0.69$  &  $0.00$ \\ 
 & $\E(Y)=0$  &  $0.00$  &  $0.08$  &  $0.00$  &  $0.71$  &  $0.31$ \\
 \midrule 
\multirow{4}{*}{2b} & $\E(X)=0$  &  $0.00$  &  $-0.03$  &  $-0.11$  &  $0.69$  &  $-0.25$ \\ 
 & $\E(W)=0$  &  $0.00$  &  $0.04$  &  $0.03$  &  $0.45$  &  $0.25$ \\ 
 & $\E(Z)=0$  &  $0.00$  &  $0.05$  &  $0.02$  &  $0.57$  &  $0.42$ \\ 
 & $\E(Y)=0$  &  $0.00$  &  $-0.01$  &  $-0.05$  &  $0.62$  &  $0.01$ \\
 \midrule 
\multirow{4}{*}{2c} & $\E(X)=0$  &  $0.00$  &  $0.14$  &  $-0.07$  &  $0.81$  &  $0.39$ \\ 
 & $\E(W)=0$  &  $0.01$  &  $0.13$  &  $-0.09$  &  $0.92$  &  $0.28$ \\ 
 & $\E(Z)=0$  &  $0.00$  &  $0.14$  &  $-0.09$  &  $0.93$  &  $0.28$ \\ 
 & $\E(Y)=0$  &  $0.01$  &  $0.17$  &  $-0.15$  &  $0.90$  &  $0.54$  \\
 \multicolumn{7}{c}{}
\end{tabular}
\begin{tabular}{lrrrrrrrr}
Example 2d & \multicolumn{7}{c}{Bias} \\
Statistic  & FMI & MI & MIRI &  CCA & ACA & Eq.~\eqref{eq:4varsb_allprobs} & Eq.~\eqref{eq:4varsb_rxprob}\\
\midrule
 $\E(X)=0$  &  $ 0.41$  &  $ 0.10$  &  0.00  &  $ 0.95$  &  $ 0.29$  &  $-0.01$  &  $-0.01$ \\ 
 $\E(W)=0$  &  $ 0.29$  &  $-0.03$  &  -0.09  &  $ 0.92$  &  $ 0.03$  &  $0.00$  &  $-0.01$ \\ 
 $\E(Z)=0$  &  $ 0.21$  &  $ 0.00$  &  -0.06  &  $ 0.78$  &  $ 0.21$  &  $-0.02$  &  $0.00$ \\ 
 $\E(Y)=0$  &  $ 0.31$  &  $ 0.10$  &  0.00  &  $ 0.87$  &  $ 0.31$  &  $-0.02$  &  $-0.01$ \\
 \bottomrule
\end{tabular}
\end{table}

\section{Conclusion} \label{sec:conclusion}

We have studied the application of multiple imputation in graphical missing data models. First, we clarified the connection between full law identifiability and multiple imputation. Theorem~\ref{thm:fulllawmi} ties the identifiability of the full law and the applicability of conditionally complete multiple imputation together. In practice, an analyst can first check if the full law is identifiable in the assumed missing data model. If this holds, a conditionally complete imputation method exists and a typical software implementation of multiple imputation could be applied. However, this does not guarantee unbiased estimates because the performance of multiple imputation depends on the chosen imputation models and other implementation details. In theory, the imputation kernel can be derived from the identifying functional of the full law. The estimation of the full law and the imputation kernels from a finite sample is a separate question that remains to be addressed in future work.

If the full law is not identifiable, a typical implementation of multiple imputation is likely to lead to biased estimates because it is theoretically not possible to devise a conditionally complete imputation method. The presented simulation examples are in agreement with Theorem~\ref{thm:fulllawmi}. The text-book style application of the \texttt{mice} software provided unbiased estimates only in the cases where the full law is identified. The unbiased estimation required that the imputation models include the response indicators. 

We introduced factorizable imputation that can be used in some cases when the full law is not identifiable. Factorizable imputation is a special imputation method in the sense that certain observed values are also imputed and weights are used in the analysis of the imputed data. Examples~\ref{fig:nocolluder}, \ref{fig:4varsfactorizes}, \ref{fig:4varsweightingmi1} and \ref{fig:4varsweightingmi2} demonstrate cases where the full law is not identified but factorizable imputation leads to unbiased estimates for target law statistics. These results are in agreement with Theorems~\ref{thm:factorizable_valid} and \ref{thm:factorizable_extended}. We studied only large-sample estimation. Investigating the small-sample properties of different estimation methods is a natural next step for  future research. 

There are also cases where the target law is identifiable but neither conditional complete multiple imputation nor factorizable imputation is applicable. In these cases, the target law can be still estimated using the identifying functionals as Examples~\ref{fig:XYindependent} and \ref{fig:4varsdonotfactorize} demonstrate. 

There is a growing understanding that the MCAR / MAR / MNAR classification could be replaced by more detailed nonparametric assumptions which are often expressed in the form of graphical models \citep{lee2023assumptions}. Clarifying the link between multiple imputation and the identifiability of the full law and the target law helps researchers in choosing the right tool for missing data analysis in this context.

\bibliographystyle{tmlr}
\bibliography{bibliography}

\appendix
\section{Detailed description of simulation setup and additional simulation results} \label{apx:fourvar}

The data for Examples~\ref{fig:4varsfactorizes}--\ref{fig:4varsdonotfactorize} are simulated as follows. Variables $X$, $W$, $Z$ and $Y$ follow multivariate normal distribution with zero means. The standard deviations and pairwise correlation coefficients can be read from the column ``Truth'' in Tables~\ref{tab:simulation_resultsXWZY_sdcor_abc} and \ref{tab:simulation_resultsXWZY_sdcor_d}. Response indicators without parents are simulated from a Bernoulli distribution where the probability of $1$ is $0.7$. Response indicators with one parent (always $X$, $W$, $Z$ or $Y$ in our examples) are simulated from a logistic regression model. Logistic regression is used also when a response indicator has multiple parents which may be both data variables and other response indicators. The linear predictor then contains high order interaction terms in addition to linear terms. The exact parametrization varies example by example and can be seen from the simulation code. 

Table~\ref{tab:simulation_resultsXWZY_sdcor_abc} presents additional simulation results for Examples~\ref{fig:4varsfactorizes}, \ref{fig:4varsweightingmi1} and \ref{fig:4varsweightingmi2}. It can be seen that  estimated standard deviations and correlations are unbiased for factorizable imputation and biased for other approaches.
Table~\ref{tab:simulation_resultsXWZY_sdcor_d} shows additional simulation results for Example~\ref{fig:4varsdonotfactorize}. The estimated standard deviations and correlations are biased for all methods except for the estimation with identifying functionals.  

\begin{table}[ht]
\centering
\caption{Simulation results on standard deviations and correlations for the examples illustrated in Figures~\ref{fig:4varsfactorizes}, \ref{fig:4varsweightingmi1} and \ref{fig:4varsweightingmi2}. The biases are reported for  factorizable multiple imputation (FMI), multiple imputation without response indicators (MI), multiple imputation with response indicators (MIRI), complete case analysis (CCA) and available case analysis (ACA).
} \label{tab:simulation_resultsXWZY_sdcor_abc}
\begin{tabular}{clrrrrrrr}
\toprule
  &  & & \multicolumn{5}{c}{Bias} \\
Example & Statistic & Truth & FMI & MI & MIRI &  CCA & ACA \\
\midrule
 \multirow{10}{*}{2a} & $\sd(X)$  &  $1.00$  &  $0.00$  &  $-0.03$  &  $0.00$  &  $-0.13$  &  $-0.04$ \\ 
 & $\sd(W)$  &  $1.00$  &  $0.00$  &  $ 0.00$  &  $-0.01$  &  $-0.14$  &  $-0.01$ \\ 
 &  $\sd(Z)$  &  $1.00$  &  $0.00$  &  $ 0.00$  &  $-0.02$  &  $-0.12$  &  $0.00$ \\ 
 & $\sd(Y)$  &  $1.00$  &  $0.00$  &  $-0.04$  &  $0.00$  &  $-0.12$  &  $-0.05$ \\ 
 & $\Cor(X,W)$  &  $0.71$  &  $0.00$  &  $-0.03$  &  $-0.01$  &  $-0.08$  &  $-0.04$ \\ 
 & $\Cor(X,Z)$  &  $0.50$  &  $0.00$  &  $-0.03$  &  $-0.01$  &  $-0.10$  &  $-0.05$ \\ 
 & $\Cor(X,Y)$  &  $0.75$  &  $0.00$  &  $-0.02$  &  $0.00$  &  $-0.06$  &  $-0.05$ \\ 
 & $\Cor(W,Z)$  &  $0.71$  &  $0.00$  &  $ 0.00$  &  $-0.01$  &  $-0.07$  &  $-0.03$ \\ 
 & $\Cor(W,Y)$  &  $0.71$  &  $0.00$  &  $-0.02$  &  $-0.01$  &  $-0.08$  &  $-0.04$ \\ 
 & $\Cor(Z,Y)$  &  $0.75$  &  $0.00$  &  $-0.01$  &  $-0.02$  &  $-0.06$  &  $-0.01$ \\ 
 \midrule
 \multirow{10}{*}{2b} & $\sd(X)$  &  $1.00$  &  $0.00$  &  $0.00$  &  $ 0.02$  &  $-0.26$  &  $-0.04$ \\ 
 &  $\sd(W)$  &  $1.00$  &  $0.00$  &  $-0.01$  &  $-0.03$  &  $-0.30$  &  $-0.04$ \\ 
 & $\sd(Z)$  &  $1.00$  &  $ 0.00$  &  $-0.02$  &  $-0.02$  &  $-0.25$  &  $-0.09$ \\ 
 & $\sd(Y)$  &  $1.00$  &  $0.00$  &  $-0.02$  &  $-0.04$  &  $-0.31$  &  $-0.12$ \\ 
 & $\Cor(X,W)$  &  $0.71$  &  $0.00$  &  $-0.04$  &  $-0.07$  &  $-0.22$  &  $-0.07$ \\ 
 & $\Cor(X,Z)$  &  $0.81$  &  $ 0.00$  &  $-0.02$  &  $-0.06$  &  $-0.14$  &  $-0.04$ \\ 
 & $\Cor(X,Y)$  &  $0.86$  &  $0.00$  &  $-0.01$  &  $-0.03$  &  $-0.11$  &  $-0.04$ \\ 
 & $\Cor(W,Z)$  &  $0.81$  &  $0.00$  &  $-0.01$  &  $-0.02$  &  $-0.16$  &  $-0.05$ \\ 
 & $\Cor(W,Y)$  &  $0.86$  &  $0.00$  &  $-0.01$  &  $-0.03$  &  $-0.15$  &  $-0.06$ \\ 
 & $\Cor(Z,Y)$  &  $0.89$  &  $ 0.00$  &  $-0.01$  &  $-0.01$  &  $-0.09$  &  $-0.04$ \\ 
 \midrule
 \multirow{10}{*}{2c} &  $\sd(X)$  &  $1.00$  &  $0.00$  &  $-0.03$  &  $ 0.09$  &  $-0.27$  &  $-0.07$ \\ 
 & $\sd(W)$  &  $1.00$  &  $0.00$  &  $-0.03$  &  $ 0.11$  &  $-0.32$  &  $-0.05$ \\ 
 & $\sd(Z)$  &  $1.00$  &  $0.00$  &  $-0.04$  &  $ 0.12$  &  $-0.35$  &  $-0.05$ \\ 
 & $\sd(Y)$  &  $1.00$  &  $0.00$  &  $-0.05$  &  $ 0.25$  &  $-0.33$  &  $-0.17$ \\ 
 & $\Cor(X,W)$  &  $0.71$  &  $0.00$  &  $-0.02$  &  $ 0.04$  &  $-0.30$  &  $-0.07$ \\ 
 & $\Cor(X,Z)$  &  $0.81$  &  $0.00$  &  $-0.01$  &  $ 0.03$  &  $-0.19$  &  $-0.04$ \\ 
 & $\Cor(X,Y)$  &  $0.86$  &  $0.00$  &  $-0.01$  &  $ 0.02$  &  $-0.14$  &  $-0.13$ \\ 
 & $\Cor(W,Z)$  &  $0.81$  &  $0.00$  &  $-0.01$  &  $ 0.03$  &  $-0.24$  &  $-0.05$ \\ 
 & $\Cor(W,Y)$  &  $0.86$  &  $0.00$  &  $-0.01$  &  $ 0.02$  &  $-0.17$  &  $-0.12$ \\ 
 & $\Cor(Z,Y)$  &  $0.89$  &  $0.00$  &  $-0.01$  &  $ 0.02$  &  $-0.14$  &  $-0.10$  \\
\end{tabular}
\end{table}

\begin{table}[ht]
\centering
\caption{Simulation results on standard deviations and correlations for the example illustrated in Figure~\ref{fig:4varsdonotfactorize}. The biases are reported for  factorizable multiple imputation (FMI), multiple imputation without response indicators (MI), multiple imputation with response indicators (MIRI), complete case analysis (CCA) and available case analysis (ACA), and for estimation using the referred equations (Eq.).
} \label{tab:simulation_resultsXWZY_sdcor_d}
\begin{tabular}{lrrrrrrrr}
 & & \multicolumn{7}{c}{Bias} \\
Statistic & Truth & FMI & MI & MIRI &  CCA & ACA & Eq.~\eqref{eq:4varsb_allprobs} & Eq.~\eqref{eq:4varsb_rxprob}\\
\midrule
 $\sd(X)$  &  $1.00$  &  $-0.09$  &  $-0.03$  &  -0.00  &  $-0.18$  &  $-0.04$  &  $ 0.00$  &  $ 0.00$ \\ 
 $\sd(W)$  &  $1.00$  &  $-0.04$  &  $ 0.03$  &  0.02  &  $-0.16$  &  $ 0.06$  &  $-0.01$  &  $ 0.00$ \\ 
 $\sd(Z)$  &  $1.00$  &  $-0.02$  &  $0.00$  &  -0.03  &  $-0.13$  &  $-0.02$  &  $ 0.01$  &  $0.00$ \\ 
 $\sd(Y)$  &  $1.00$  &  $-0.05$  &  $-0.03$  &  -0.00  &  $-0.14$  &  $-0.05$  &  $ 0.02$  &  $ 0.00$ \\ 
 $\Cor(X,W)$  &  $0.71$  &  $-0.03$  &  $0.00$  &  0.01  &  $-0.09$  &  $ 0.00$  &  $0.00$  &  $ 0.00$ \\ 
 $\Cor(X,Z)$  &  $0.50$  &  $-0.03$  &  $ 0.00$  &  0.00  &  $-0.12$  &  $-0.08$  &  $ 0.02$  &  $ 0.00$ \\ 
 $\Cor(X,Y)$  &  $0.75$  &  $-0.03$  &  $-0.01$  &  -0.00  &  $-0.08$  &  $-0.05$  &  $ 0.01$  &  $ 0.00$ \\ 
 $\Cor(W,Z)$  &  $0.71$  &  $-0.02$  &  $ 0.02$  &  -0.00  &  $-0.08$  &  $-0.05$  &  $ 0.00$  &  $ 0.00$ \\ 
 $\Cor(W,Y)$  &  $0.71$  &  $-0.03$  &  $-0.01$  &  -0.01  &  $-0.10$  &  $0.00$  &  $0.00$  &  $ 0.00$ \\ 
 $\Cor(Z,Y)$  &  $0.75$  &  $-0.01$  &  $0.00$  &  -0.02  &  $-0.07$  &  $-0.02$  &  $ 0.00$  &  $0.00$  \\
 \bottomrule
\end{tabular}
\end{table}

\end{document}